\documentclass[a4paper, 12pt]{article}     
\usepackage{amsmath}
\usepackage{ stmaryrd }
\usepackage{amsthm}
\usepackage{url,a4wide}
\usepackage{amsmath,xspace,color,pict2e}
\usepackage{amsfonts,amssymb}
\usepackage[UKenglish]{babel}
\usepackage{macros}
\usepackage{tikz-cd}
\usepackage{enumerate}
\usepackage{mathtools}

\begin{document}

\title{Weights and characters for affine Hecke algebras}

\author{Maximilien J. Mackie \\
}
 \date{}

\maketitle

\begin{abstract}
Weights play an essential role in the  classification of simple modules for affine Hecke algebras. In this paper we show the extent to which weights determine simple modules and relate this to induction and restriction along isogenies. We also describe the behaviour of hermitian duals and unitarity under  induction and restriction along isogenies.
\end{abstract}

\section{Introduction}
    The representation theory of affine Hecke algebras has a central role in the  Langlands program via the equivalence of categories with Bernstein blocks. An affine Hecke algebra $\He = \He(\R)$ attached to a based root datum $\R = (X , R , Y , R^\vee, \Delta)$ contains a large commutative subalgebra $\Ab = \C[X]$ (denoted $\mathcal O$ by Lusztig \cite{Lusztig})  which acts on finite length modules $V$ by a set of weights $\Omega(V) \subset \Hom(X, \C^\times)$.  The character $\Theta(V)$ of $V$ is the data $\Omega(V)$ along with the dimensions of  the weight spaces. This paper studies the extent to which $\Omega(V)$ and $\Theta(V)$ determine $V$,  how they transform under induction and restriction along central morphisms (see Definition \ref{def: centralMorphism}), and as an application we present optimised algorithms for computing presentations of simple modules and composition series of finite length modules.
    
     First, we show that for quasi-simply-connected affine Hecke algebras (Definition \ref{def: qsc}) that  simple modules 
     are uniquely determined by their weights, and semisimplifications of finite length modules are uniquely determined by their characters. 
     This 
     generalises results of Evens and Mirkovi\'c \cite{EM} and Barbasch and Ciubotaru \cite{BC2013}. The result on semisimplifications has also been  shown by Antor and Okada \cite{AntorOkada}.

   Second, we study induction and restriction along central morphisms.
We  show for arbitrary isogeny class that weights (resp. characters) determine simple modules (resp. semisimplifications of finite length modules) up to explicit outer automorphisms of the Hecke algebra. 

     Third, we  show  induction and restriction along isogenies
     commutes with taking hermitian duals and preserves unitarity. This generalises a result of Delorme and Opdam  \cite{DelormeOpdam}.

\paragraph{Acknowledgement}
We thank Dan Ciubotaru for suggesting the problems investigated in this paper and for  helpful explanations, and (in alphabetical order) Stefan Dawydiak, Mick Gielen, Emile Okada, and Maarten Solleveld for useful discussions.

\section{Preliminaries} \label{sec: prelim}

Fix a reduced  based root datum $\R = (X, R, Y , R^\vee, \Delta)$ with finite Weyl group $W$, affine Weyl group $W^a = W \ltimes \Z R$ with simple reflections $S_a$ and extended affine Weyl group $\widetilde{W} = W \ltimes X$.

Write $\He = \He(\R) = \He(\R, \lambda, \lambda^* , q) $ for the (specialised) affine Hecke algebra  associated to $\R$ with parameters $\lambda, \lambda^* \colon \Delta \to \Re \backslash \{ 0 \}$ and $q \in \Re_{> 1}$. The parameters $\lambda, \lambda^*$ are constant on $\widetilde{W}$-orbits  and
\begin{equation} \label{eq: parameters}
 \lambda(\alpha) = \lambda^*(\alpha) \text{ if }  \alpha^\vee \not\in 2Y.
\end{equation}
We impose further
\begin{equation} \label{eq: parametersNE}
    \lambda(\alpha)  +  \lambda^*(\alpha) \neq 0 \text{ if }  \alpha^\vee \in 2Y.
\end{equation}

We say $\R$ and its associated affine Hecke algebra $\He ( \R )$ are
simply-connected (resp. adjoint)  if  $Y / \Z R^\vee$ (resp. $X / \Z R$) is torsion-free. It is well-known that if $\R$ is irreducible then $\alpha^\vee \in 2Y$ occurs if and only if $\alpha$ is a short root in  an adjoint root datum $\R_B$ of type $B$.
More generally, Lusztig \cite[Lemma 1.7]{Lusztig} showed that a reduced based root datum $\R$ decomposes into a direct sum of based  root data $\R_{s} \oplus  \R_{B,1} \oplus \dots \oplus \R_{B,n} $ such that
\begin{enumerate}
    \item $\alpha^\vee \not\in 2Y$ for all $\alpha \in R_s$, and
    \item  there is a unique $\alpha \in \Pi_{B,i}$ such that $\alpha^\vee \in 2Y$.
\end{enumerate}
Such summands  are called primitive.  This gives a decomposition of  $\He(\R)$ into a finite tensor product of primitive affine Hecke algebras $\He(\R_i)$, and furthermore as the simple $\He(\R)$-modules are finite dimensional they decompose as tensor products of simple $\He(\R_i)$-modules. In arguments we may therefore  assume $\R$ is primitive.

The affine Hecke algebra $\He$ has  linear decomposition $\He(W) \otimes_\C \Ab$ where $\He(W)$ is the finite-dimensional Hecke algebra and $\Ab = \C[X]$. Identify $\He(W)$ and $\Ab$ with their images in $\He$ and write $T_w $ for the basis element of $\He(W)$ associated to $w \in W$ and $\theta_x $ similarly. 

Let $T  = \Hom(X, \C^\times) = Y \otimes_\Z \C^\times$  and $\afr = Y \otimes_\Z \Re$ with complexification $\afr_\C = Y \otimes_\Z \C$. 
    Write $\Irr(\He)$ for the set of isomorphism classes of simple left $\He$-modules $V$. These are finite dimensional so  upon restriction to $\Ab$ admit a weight space decomposition $V = \bigoplus_{s \in T} V_s$ where $V_s$ is the  generalised  $s$ weight space. The set of $s$ such that $V_s \neq 0$ are the $\Ab$-weights $ \Omega_\Ab(V)$ of $V$. Their real split parts are in bijection with $\Lambda_\Ab (V) = \{ \log |s| : s \in \Omega_\Ab(V) \}$ via the exponential map. We omit the subscript where unambiguous.

\begin{lemma}
Let  $U \subset V$ be $\He$-modules and $s \in \Omega(V)$. Then  $\Omega(U) \subset \Omega(V)$, $U_s = U \cap V_s$, and $(V / U)_s = V_s / U_s$.     
\end{lemma}
    \begin{proof}
        The first two are clear, the third follows from exactness of $V \mapsto V_s$.
    \end{proof}

    The centre $\Ab^W$ of $\He$ will act by a central character $\cc(V)$ (often called an infinitesimal character). This is a $W$-orbit in $T$ containing $\Omega(V)$. 
    Fix a base of simple roots $\Delta  \subset R$. We say $V$ is tempered if 
    \[
    \Lambda(V) \subset \afr^- = \big\{ \sum_{\alpha \in \Delta} \lambda_\alpha \alpha^\vee : \lambda_\alpha \leq 0 \big\}.
    \]

The Iwahori-Matsumoto involution $\IM \colon \He \to \He$ is the $\C$-linear algebra automorphism defined by
\[
\IM(T_w) = -q T_w^{-1}, \: \IM(\theta_x) = \theta_x^{-1}.
\]
This induces an automorphism on $\He$-modules preserving irreducibility and satisfying $\IM(V_s) = (\IM(V))_{s^{-1}}$. We say $V$ is antitempered if $\IM(V)$ is tempered. 

The following notation is from \cite{Solleveld2010}.  For $P \subset \Delta$ let 
\begin{equation*}
\begin{array}{l@{\qquad}l}
 R_P = R \cap \Z P &  R_P^\vee = R^\vee \cap \Z P^\vee ,\\
 \afr_P = \Re \otimes P^\vee &  \afr^P = ( \afr^*_P )^\perp ,\\
 \afr^*_P = \Re \otimes P &  \afr^{P*} = ( \afr_P )^\perp  ,\\
X_P = X \big/ \big( X \cap (P^\vee )^\perp \big) &
X^P = X / (X \cap  \Q P ) , \\
Y_P = Y \cap  \Q P^\vee & Y^P = Y \cap P^\perp , \\
T_P = {\Hom}_{ \Z} (X_P,  \C^\times ) &
T^P = {\Hom}_{ \Z} (X^P,  \C^\times ) , \\
 \R_P = ( X_P ,R_P ,Y_P ,R_P^\vee ,P) &  \R^P = (X,R_P ,Y,R_P^\vee ,P) , \\
 \afr^{P++}  =  \{ \mu \in  \afr^P : \la \alpha, \mu \ra > 0 , 
  \forall \alpha \in \Delta \setminus P \} , & T^{P++}  = \exp ( \afr^{P++}), \\
  \afr_P^- =  \{ \sum_{\alpha \in P} c_\alpha \alpha^\vee : c_\alpha \leq 0 \}, &
    T_P^- = \exp (  \afr_P^-).
\end{array} 
\end{equation*}

The root subdata $\R^P $ and $\R_P$ both have Weyl group $W_P$. The associated affine Hecke algebras are $\He^P = \He(\R^P)$ with linear decomposition $\He(W_P) \otimes_\C \Ab$ and $\He_P = \He(\R_P)$ with linear decomposition $\He(W_P) \otimes_\C \Ab_P$ where $\Ab_P = \C[X_P]$.  Write $x \mapsto x_P$ for the quotient $X \to X_P$ and similarly $x \mapsto x^P$ for $X \to X^P$. The canonical 
surjection $\He^P \to \He_P$ is an algebra homomorphism. 

For $t \in T^P$ there is an algebra automorphism
\begin{equation}  \label{eq: twist}
\phi_t \colon \He^P \to \He^P, \quad \phi_t (\theta_x T_w) = t (x^P) \theta_x T_w
\end{equation}

Given an $\He_P$-module $\pi$ we write $\pi \circ \phi_t$ for the inflated $\He^P$-module twisted by $\phi_t$. 
It follows 
\begin{equation*} 
    \Omega(\pi \circ \phi_t) = \Omega(\Inf_{\He_P}^{\He^P} \pi) \otimes \Inf_{X^P}^{X}t.
\end{equation*}

\begin{lemma} \label{inducedWeights}
    Let $\pi$ be an $\He^P$-module. The weights of $\Ind{\He^P}{\He} \pi$ are $W^P \cdot \Omega(\pi)$ where $W^P$ is the set of minimal length coset representatives for $W / W_P$. The $\He^P$-submodule $1 \otimes \pi$ has weights $\Omega(\pi)$.
\end{lemma}

\begin{proof} 
     Let $W^P$ be a set of minimal length coset representatives for $W/W_P$. These are characterised as $\{w \in W : l(w s_\alpha) > l(w) \text{ for all } \alpha \in P \}$ \cite[Proposition 1.10]{Humphreys_1990}. 
  There is a basis $v_i^{s_j}$ of generalised weight vectors for $\pi $ such that $(\theta_x - s_j(x))^iv_i^{s_j} = 0$. Let $B$ be the ordered basis  $B = (v_1^{s_1}, v_2^{s_1}, \dots, v_1^{s_2}, v_2^{s_2}, \dots)$. With this basis each $\theta_x$ acts upper diagonally.
  Extend this to the ordered basis $(T_{w_1} \otimes B, T_{w_2} \otimes B, \dots)$ of the underlying space
  \[\Ind{\He^P}{\He} \pi = \bigoplus_{w \in W^P} T_w \otimes \pi\]
 where $l(w_i) \leq l(w_{i+1})$. From the algebra relations it follows that $\theta_x$ acts on this basis also upper diagonalisably with weights $\{w(s) : w \in W^P, s \in \Omega(\pi ) \} $.
\end{proof}

\subsection{Langlands classification} 
 We review the Langlands classification  \cite{Solleveld2010, Evens} for simple $\He(\R)$-modules $L(P, \pi, t)$ in terms of standard modules $M(P, \pi, t) = \Ind{\He^P}{\He} (\pi \circ \phi_t)$. The standard argument for uniqueness is \cite[\S 2.7]{Evens} which shows  maximal weights in $M(P, \pi, t) $ occur only in the linear subspace $1 \otimes (\pi \circ \phi_t)$. In Proposition \ref{prop: Langlands class} we strengthen this for use in later arguments.

Let $\{ \omega_1^\vee , \dots , \omega_n^\vee \} $ be a basis of $R^\vee \otimes \Re \subset \afr$ 
dual to $\Delta$.  
By a lemma of Langlands \cite[IV, 6.11-6.13]{BorelWallach}, for each $\lambda \in \afr$ there is a unique subset $F_\lambda \subset \{1, \dots, n\}$ such that 
\begin{equation} \label{eq: Langlands}
\lambda = \lambda^\Delta + 
\sum_{j \notin F_\lambda} c_j \omega_j^\vee + \sum_{i \in F_\lambda} d_i \alpha_i^\vee
\end{equation}
with $\lambda^\Delta \in \afr^\Delta$, $ c_j >0$, $ d_i \leq 0$. Let 
$\lambda^0 = \lambda^\Delta + \sum_{j \notin F_\lambda} c_j \omega_j^\vee $ and $P_\lambda = \{ \alpha_i \in \Delta : i \in F_\lambda \}$. 
Then \[ 
\lambda^0 \in \afr^{P_\lambda ++}, \quad \lambda - 
\lambda^0  \in \afr_{P_\lambda}^-\] 
giving $\lambda \in \afr_{P_\lambda}^- \oplus \afr^{P_\lambda ++}$.

Order $\afr$ by $\lambda \geq \mu$ 
if $\lambda - \mu \in \Re_{\geq 0} R^{\vee +}$.
Then $\lambda \geq \mu$ implies $\lambda^\Delta = \mu^\Delta$ and $\lambda^0 \geq \mu^0$ by \cite[(2.13)]{RamKril}.

\begin{lemma}[Langlands classification]
    Let $\pi \in \Irr(\He_P)$ be  tempered and $t \in T^{P}$ such that $|t| \in T^{P++}$. Then the $\He$-module $M(P, \pi, t) = \Ind{\He^P}{\He} (\pi \circ \phi_t)$ has a unique irreducible quotient $L(P, \pi, t)$ and every irreducible $\He$-module has this form. If $L(P, \pi, t) \cong L(P', \pi', t')$ then $P = P'$ and $\pi \circ \phi_t \cong \pi' \circ \phi_{t'}$. 
\end{lemma}

The construction is as follows: starting with an irreducible $\He$-module $L$, 
choose $s \in \Omega(L)$ such that $\lambda = \log |s| \in \Lambda(L)$ has $\nu = \lambda^0$ maximal. Let $P = P_\lambda$ and take $\He^P$-module $\rho = \He^P  v_s$ where $v_s$ is an $s$-weight vector. Choose an irreducible submodule of $\rho$; Solleveld \cite{Solleveld2010} shows these are  of the form $\pi \circ \phi_t$ with $\pi$ an irreducible $\He_P$-module and $\log|t| = \nu$. Then 
$L = L(P, \pi, t)$ and the $\Ab_P$-weights of $\pi$ satisfy 
\[
\Lambda(\pi) \subset \{  \lambda'|_{\afr_P} : \lambda' \in \Lambda(L), \:  F_{\lambda'} = F_{\lambda} , \: \lambda'|_{\afr^P} = \nu \}. 
\]

The following instance of \cite[Lemma 2.2.3]{Solleveld2010} will also be needed.

\begin{lemma} \label{lemma: t=t'}
    Let $\pi$ be a tempered $\He_P$-module and $t, t' \in T^P$. If $\Omega(\pi \circ \phi_t) \cap \Omega(\pi \circ \phi_{t'}) \neq \emptyset$ then there is a tempered $\He_P$-module $\pi'$ satisfying $\pi' \circ \phi_t = \pi \circ \phi_{t'}$. 
\end{lemma}
\begin{proof}
    Let $s \otimes t = s' \otimes t' \in \Omega(\pi \circ \phi_t) \cap \Omega(\pi \circ \phi_{t'})$. View $s,s' \in T_P$  as characters $X \to \C^\times$ trivial on $X \cap P^{\vee\perp}$ and view $t,t' \in T^P$ as characters $X \to \C^\times$ trivial on $X \cap \Q P$. Then $t = t'$ on $X \cap P^{\vee\perp}$ 
    so $u = t' t^{-1} \colon X \to \C^\times$ factors through the finite group
    \[
    u \colon X /( X \cap P^{\vee\perp} +X \cap \Q P) \to \C^\times.
    \]
    and hence $\psi_u \colon \He_P \to \He_P, \theta_{x_P} T_w \mapsto u(x_P) \theta_{x_P} T_w$ is a well-defined algebra automorphism. The $\He_P$-module $\pi' = \pi \circ \psi_u$ is irreducible and  $|u(X)| = 1$ so $\Lambda(\pi') = \Lambda(\pi)$ so $\pi'$ is tempered, and
    \begin{equation*}
        (\pi') \circ \phi_{t} (\theta_x T_w) =  \pi(\theta_x T_w) t'(x) t^{-1}(x) t(x) = \pi  \circ \phi_{t'} (\theta_x T_w).
    \end{equation*}
\end{proof}

\begin{proposition} \label{prop: Langlands class}
    Let $\pi \in \Irr(\He_P)$ be  tempered and $t \in T^{P}$ such that $|t| \in T^{P++}$. The weights  occurring in the linear quotient $M(P, \pi, t) / 1 \otimes (\pi \circ \phi_t)$  are disjoint from $\Omega(\pi \circ \phi_t)$.
\end{proposition}
\begin{proof}
    From Lemma~\ref{inducedWeights},  $\Omega(M(P, \pi, t)) = \{ w (s \otimes t) : w \in W^P, s \in \Omega(\pi) \}$ and the quotient $M(P, \pi, t) / 1 \otimes (\pi \circ \phi_t)$ has weights  $\{ w (s \otimes t) : 1 \neq w \in W^P , s \in \Omega(\pi) \}$. 
   Note that $\log |ws| = w \log |s|$. Let \[\lambda = \log|s|, \: \lambda' = \log|s'| \in \Lambda_P(\pi) \subset \afr_P^-\] and $\nu = \log|t| \in \afr^{P++}$. 
  If $w(s \otimes t) = s'\otimes t$ for some $ s' \in \Omega(\pi)$, 
  then the image of $ w(\lambda + \nu) $ in $\afr^P$ is $\nu$. We will show this does not occur.

   Write $\lambda = \sum_{\alpha_i \in P } d_i \alpha_i^\vee $ and 
   $\nu = \sum_{\alpha_j \not\in P } c_j \omega_j^\vee$ with all $d_i \leq 0$ and $c_i >0$. Let $\rho^P = \sum_{\alpha_i \not\in P} \omega_i$. There is a direct sum decomposition $\afr = \afr_P \oplus \afr^P$ and $\la \cdot , \rho^P \ra$ is trivial on $\afr_P$, so 
   it suffices to show $\la w(\lambda + \nu) , \rho^P \ra \neq \la \nu , \rho^P \ra$.

   As $w \in W^P$ it follows $wP \subset R^+$ so $\la w \lambda , \rho^P \ra \leq 0 = \la \lambda , \rho^P \ra$.
   Let $w = s_1 \dots s_n$ with $s_i = s_{\alpha_i}$. Then one sees by induction 
   \[
   w \omega_j^\vee = \omega_j^\vee  - \sum_i s_1 \dots s_{i-1}\alpha_i^\vee \la \omega_j^\vee , \alpha_i \ra = \omega_j^\vee  - \sum_{\alpha_i = \alpha_j} s_1 \dots s_{i-1}\alpha_i^\vee
   \]
   and furthermore by {\cite[\S 1.7]{Humphreys_1990}}
   \begin{equation} \label{eq: root systems}
        \{ \alpha \in R^+ : w\alpha \in R^- \} = \{ \alpha_n , s_n \alpha_{n-1}, \dots , s_n \dots s_2 \alpha_1 \}
    \end{equation}
    with all terms on the right distinct.
   Applying this to $w^{-1}$ gives $\la w \omega_j^\vee , \rho^P \ra \leq  \la \omega_j^\vee , \rho^P \ra$ for all $j$. We show $\la w \omega_j^\vee , \rho^P \ra <  \la \omega_j^\vee , \rho^P \ra$ for some $j$. As $ wP > 0$ it follows from (\ref{eq: root systems}) that $\alpha_n \not\in P$. Let $j$ be the lowest index such that $\alpha_j \not\in P$. Then $\alpha_i \in P$ for all $i < j$ so $s_1 \dots s_{j-1}\alpha_j^\vee + \afr_P = \alpha_j^\vee + \afr_P \neq 0 + \afr_P$. For this $j$, $\la w \omega_j^\vee , \rho^P \ra < \la \omega_j^\vee , \rho^P \ra$ and $\omega_j^\vee$ occurs in $\nu$. Hence $\la w \nu , \rho^P \ra < \la \nu , \rho^P \ra$.
   Combining these gives $\la w (\lambda + \nu) , \rho^P \ra < \la \nu , \rho^P \ra$ so $w(\lambda + \nu) + \afr_P \neq \nu + \afr_P$.
\end{proof}

\begin{corollary} \label{cor: maximal weights}
   Let $\rho^P$ be as in the proof of Proposition \ref{prop: Langlands class}, viewed as a map $\Omega(M(P,\pi,t)) \to \Re$.  Then $\Omega(\pi \circ \phi_t)$ is the set of points on which $\rho^P$ attains its maximum $\la \nu , \rho^P \ra$. 
\end{corollary}

\begin{definition}
  A  finite dimensional $\He$-module $\pi$ is \emph{constructed} if there exists a basis $v_1, \dots , v_n$ and  matrices $\pi(h) \in \mathrm{Mat}_n(\C)$ for $h \in G$ where $G =  \{ T_w : w \in W \} \cup \{\theta_{x_i} \colon x_i \text{ generators of } X \}$ is a generating set of $\He$. 
\end{definition}
 We assume that if two modules $A \subset B$ are constructed then there is a finite algorithm to construct $B / A$, and also there are available  finite algorithms for handling root data, Weyl groups, and their associated linear algebra. The  weights of a constructed module are computed from the matrices $\pi(\theta_{x_i}) $. 

\begin{proposition} \label{prop: algorithm}
 Suppose the  tempered $\He_P$-module $\pi$ is constructed.  Then there is a finite algorithm to construct $L(P, \pi,t)$.
\end{proposition}
\begin{proof}
    Let $U = \pi \circ \phi_t \in \Irr (\He^P)$. Then $M = M(P, \pi , t) = \He \otimes_{\He^P} U$ has 
    an explicit basis with respect to which the matrices for $h \in G$ are constructed. Diagonalising for the action of the generators $\theta_{x_i}$ gives a basis of generalised weight vectors.  It suffices to construct the maximal submodule $V$ of $M$. 
    
    As $M$ is generated by $1 \otimes U$ and $U$ is a simple module for $\He^P \subset \He$, $V \cap (1 \otimes U )= 0$.  By Proposition \ref{prop: Langlands class} this is equivalent to $\Omega(V) \cap \Omega(U) = \varnothing$ so $V$ is contained in the span of the generalised weight vectors for $\Omega(M) \backslash \Omega(U)$. By Corollary \ref{cor: maximal weights} it suffices to compute the subset of $\Omega(M)$ which does not maximise  $\rho^P$. Let $V_0 $ be their linear span inside $M$ and 
    \[
    V_{k+1} = V_k \cap \bigcap_{h \in G} \pi(h)^{-1} (V_k).
    \]
    This is a decreasing sequence of finite dimensional spaces, hence stabilises at $V_\infty = V_\infty \cap \bigcap_{h \in G} \pi(h)^{-1} (V_\infty)$ which is closed under $\He$. By induction $V \subset V_\infty$ hence  by maximality $V = V_\infty$. The quotient is $L(P, \pi, t)$.
\end{proof}

\section{Weights and characters} 
A result of \cite{EM} is that weights uniquely determine simple modules for a family of affine Hecke algebras with equal nonzero parameters.
A similar result of
\cite{BC2013} is that for graded affine Hecke algebras with nonzero parameters, the graded analogue of $\Ab$-characters of simple modules are linearly independent. In this section we generalise both results to quasi-simply-connected affine Hecke algebras (see Definition \ref{def: qsc}).

\begin{lemma}[{\cite[Theorem 2.2]{kato1981irreducibility}, see also \cite[\S 3]{Solleveld_2021}}] \label{lemma: Kato original}
    Let $t \in T$ and
\[
c_\alpha = \frac{\left(\theta_\alpha - q^{-\frac{\lambda^*(\alpha) + \lambda(\alpha)}{2}}\right)\left(\theta_\alpha + q^{\frac{\lambda^*(\alpha) -  \lambda(\alpha)}{2}}\right)}{(\theta_\alpha - 1)(\theta_\alpha + 1)}.
\]
The induced module $\Ind{\Ab}{\He} t$ is irreducible if and only if $c_\alpha(t) \neq 0$ for all $\alpha \in R$ and $\Stab_W(t)$ is generated by $\{s_\alpha : \alpha \in R, s_\alpha(t) =t, c_\alpha^{-1}(t) =0 \}$.
\end{lemma}

Recall  $\He(\R)$ has nonzero parameters $\lambda, \lambda^* \colon R \to \Re \backslash \{0\}$ satisfying (\ref{eq: parameters}) and (\ref{eq: parametersNE}).

\begin{definition} \label{def: qsc}
    An affine Hecke algebra and its associated root datum $\R$ with parameter functions $\lambda, \lambda^*$ are \emph{quasi-simply-connected} if each primitive summand  of $\R$ is simply-connected or is $ \R_B$ and $\lambda(\alpha) \neq \lambda^*(\alpha)$ for the unique $\alpha \in \Pi$ such that $\alpha^\vee \in 2Y$.
\end{definition}

\begin{lemma} \label{lemma: qsc induction}
    Let $\He = \He(\R)$ be quasi-simply-connected. Let $\Delta$ be a base for $\R$ and $P \subset \Delta$. Then $\He^P$ and $\He_P$ are quasi-simply-connected.
\end{lemma}
\begin{proof}
   We first handle the simply-connected case. It is clear that $Y / \Z P$ has torsion if $(Y \cap \Q P^\vee) / \Z P$ has torsion, and conversely if $y \in Y / \Z P$ is torsion then $ny \in \Z P$ so $y \in Y \cap \Q P^\vee$. Hence these are equivalent and imply that $Y / \Z \Delta$ has  torsion, so for simply-connected summands of $\R$ it follows that $\R^P$ and $\R_P$ are simply-connected.
    
    Suppose $\R=\R_B$ is not simply-connected. It is well-known for type $B$ that  $Y / \Z R^\vee$ has at most one torsion element, associated to  $\alpha \in \Pi$ such that $\alpha^\vee \in 2Y$. By assumption $\lambda(\alpha) \neq \lambda^*(\alpha)$. If $\alpha \not\in P$ then the root data $\R^P$ and $\R_P$  are a sum of simply-connected components of type $A$. If $\alpha \in P$ there will be one component of type $B$ containing $\alpha$ and possibly additional simply-connected components of type $A$. In the component containing $\alpha$  we have $\alpha^\vee \in 2 Y_P \subset 2Y$ and $\lambda(\alpha) \neq \lambda^*(\alpha)$ so $\He^P$ and $\He_P$ are quasi-simply-connected.
\end{proof}

The following lemma is stated in {\cite[Lemma 4.5]{lusztigUnipotent1}} which attributes the simply-connected case to Steinberg and leaves $B_n$ as an exercise. 

\begin{lemma} \label{lemma: reflectionGroup}
    If $\R$ is quasi-simply-connected  and $t \in T$ then $\Stab_W(t)$ is generated by its reflections.
\end{lemma}
\begin{proof}
     It suffices to consider $\R$ primitive. First suppose $\R$ is simply-connected. Let $w \in \Stab_W(t)$. 
    The exponential map $\exp \colon \afr_\C \to T, \tau \mapsto \exp(2 \pi i \la \tau , - \ra)$ is surjective with kernel $Y$ and $W$-equivariant. Let $\exp(\tau) = t$. Then $\tau - w\tau \in Y$ and as $w \in W$ 
    it follows that $\tau - w\tau \in Y \cap \Q R^\vee = \Z R^\vee$ as $\R$ is simply-connected.
    
    Let $W^{a} = W \ltimes \Z R^\vee$ be the affine Weyl group and $(w, \lambda) \in W^a$. This acts on $\afr$ and $\afr_\C$ via $(w, \lambda)(y) = w(y) + \lambda$. This is a simple transitive action on the set of alcoves of $W^a$ in $\afr$  and one obtains a reduced expression for $(w, \lambda) \in W^a$ by considering the walls $H_{\alpha , n} = \{ y \in \afr : \la \alpha , y \ra = n  \}$ separating the fundamental alcove $A_0$ from $(w,\lambda)A_0$
    \cite[\S1.5]{Humphreys_1990}.

 Recall $(w, \tau - w \tau) \in \Stab_{W^a}(\tau)$. First suppose  $\tau  \in \afr$, so $\tau \in \overline{A} \cap \overline{(w, \tau - w \tau)A}$ where $A$ is an alcove containing $\tau$. Choose a gallery of alcoves connecting $A$ to $(w, \tau - w \tau)A$  all containing $\tau$ in their closure. Adjacent alcoves are separated by walls containing $\tau$ so the associated affine reflections fix $\tau$.
    This gives a decomposition of $(w, \tau - w \tau) = s_{\alpha_1 , n_1} \dots s_{\alpha_k ,  n_k}$ into affine reflections $s_{\alpha , n} = (s_\alpha , n \alpha^\vee)$ each fixing $\tau$. Projecting $W^a \to W$ gives the decomposition $w = s_{\alpha_1} \dots s_{\alpha_k}$ with each $s_{\alpha_i} (\tau) = \tau + \Z R^\vee$ hence $s_{\alpha_i} \in \Stab_W(t)$. 
    
    If $\tau \not\in \afr$ then decompose $\tau = \tau_\Re + i \tau_\Im$. As $\tau - w \tau  \in Y \subset \afr$ it follows that $\tau - w \tau = \tau_\Re - w \tau_\Re$ and also $w \in \Stab_W(\tau_\Im)$. The previous argument applied to $\tau_\Re $ shows $\Stab_W(\exp \tau_\Re)$ is generated by reflections. The same alcove argument applied to the root system associated to the reflections in $\Stab_W(\exp \tau_\Re)$ shows the stabiliser of $\tau_\Im$ in $\Stab_W(\exp \tau_\Re)$ is also generated by reflections.

    This leaves the case  $\R = \R_B$ is adjoint $B_n$ so $X \cap \Q R = \Z R$ with free basis given by the short roots  and $W = S_n \ltimes (\Z / 2)^n$ acts by signed permutations on these. The Weyl group only acts on $X \cap \Q R$ so we can assume $\R$ is semisimple and $T \cong (\C^\times)^n $. 
    Let  $t = (z_1 , \dots , z_n)$ and $w \in \Stab_W(t)$. If $z_1 \not\in \{z_1^{-1}, z_2^{\pm1} , \dots , z_n^{\pm1} \}$ then the $S_n$ factor of $w$ fixes the first component of $T$ and we are done by induction. Otherwise if $w$ sends $z_i^{\pm 1} $ to $z_1$ then precomposing by the signed reflection $s \in \Stab_W(t)$ of components $1$ and $i$ gives $ws \in \Stab_W(t)$ fixing the first component and again we are done by induction.
\end{proof}

Recall conditions (\ref{eq: parameters}) and (\ref{eq: parametersNE}) for the parameter functions $\lambda, \lambda^*$ are assumed.

\begin{theorem} \label{thm: EM}
     Suppose $\He$ is quasi-simply-connected and 
     $V \in \Irr(\He)$ is tempered and antitempered. Then $\cc(V)$ uniquely determines $V$.
\end{theorem}

\begin{proof}
   Choose $t \in \cc(V)$. By Lemma~\ref{inducedWeights} with $P$ empty,
   $ \Omega(\Ind{\Ab}{\He} t) = W \cdot t = \cc(V)$ and by Frobenius reciprocity there are  nonzero maps $\Ind{\Ab}{\He} t \to V$ and $\Ind{\Ab}{\He} t \to V'$ so it suffices to show  $\Ind{\Ab}{\He} t$ is irreducible.  The first condition of Lemma~\ref{lemma: Kato original} is always satisfied as $q \in \Re_{>1}$ and as $V$ is tempered and antitempered $|t(\alpha)|=1$ for all $\alpha \in R$ so zeros of the numerator 
   can occur only when $t(\alpha) = \pm 1$.

    By Lemma \ref{lemma: reflectionGroup}, $\Stab_W(t)$ is generated by simple reflections $s_\alpha$.
   We show $c_\alpha^{-1}(t) = 0$. If $\alpha$ is in a simply-connected component then  as $\lambda(\alpha) = \lambda^*(\alpha) \neq 0$ it suffices to show $t(\alpha) = 1$. 
    As $s_\alpha(t) = t$ it follows that $t(\alpha)^{\la x , \alpha^\vee \ra } = 1$ for all $x \in X$. As $Y / \Z R^\vee$ is torsion-free we can choose a basis of $Y$ containing $\alpha^\vee$ and hence there exists $x \in X$ such that $\la x , \alpha^\vee \ra = 1$.
If $\alpha$ is in an adjoint $B_n$ component then from $t(\alpha)^{\la x , \alpha^\vee \ra } = 1$ we conclude $t(\alpha) = \pm 1$. We may construct $x$ as before except when $\alpha^\vee \in 2 Y$ in which case  $\lambda(\alpha)  \neq \pm \lambda^*(\alpha)$ ensures that $c_\alpha^{-1}(t) = 0$ in either case.
\end{proof}

\begin{remark}
    A simply-connected affine Hecke algebra with suitable parameters 
    will occur as the affine Hecke algebra of a complex group with simply-connected derived group. For this family of affine Hecke algebras Theorem~\ref{thm: EM} is proved in {\cite[\S 5]{EM}} using the unipotent local Langlands correspondence, from which they deduce the following corollary {\cite[5.5 Theorem]{EM}}. We include their argument for completeness.
\end{remark}

\begin{theorem} \label{thm: weights}
    Suppose $\He$ is quasi-simply-connected and 
     $V\in \Irr(\He)$. Then $\Omega(V)$ uniquely determines $V$.
\end{theorem}
\begin{proof}
    Let $V = L(P, \pi , t)$ and $V' = L(P', \pi', t')$ have $\Omega(V) = \Omega(V')$. By the Langlands classification $\Omega(V)$ determines $P$ hence by Corollary \ref{cor: maximal weights} $\Omega(\pi \circ \phi_t) = \Omega(\pi' \circ \phi_{t'})$. If $P \neq \Delta$ then by Lemma \ref{lemma: qsc induction} it follows by induction that $\pi \circ \phi_t = \pi' \circ \phi_{t'}$. 
    If $P = \Delta$ then $V$ and $V'$ are tempered. Applying $\IM$ and the previous argument we reduce to the case that $V$ and $V'$ are tempered and antitempered, in which case the result follows from Theorem \ref{thm: EM}.
    \end{proof}

For $V$ an $\He$-module and $s \in \Omega(V)$, let $[s : V] = \dim V_s$ be the dimension of the  generalised $s$-weight space.
\begin{lemma} \label{lemma: weights mult.}
    Let $s$ be a weight of the irreducible tempered $\He_P$-module $\pi$. Then 
    \[
    [s \otimes t : L(P, \pi, t)] =  [s \otimes t : M(P, \pi, t)] = [s : \pi]
    \]
    where $(s \otimes t) (x) = s(x_P) t(x^P)$.
\end{lemma}
\begin{proof}
    By Frobenius reciprocity there is a nonzero $\He^P$-morphism $\pi \circ \phi_t \to \Res{\He^P}{\He} L(P, \pi, t)$.
  As $\pi \circ \phi_t$ is irreducible, $[s \otimes t : L(P, \pi, t)] \geq [s \otimes t : \pi \circ \phi_t]$ and by definition $[s \otimes t : \pi \circ \phi_t] = [s  : \pi ]$ so it suffices to show $[s \otimes t : \pi \circ \phi_t] \geq [s \otimes t : M(P, \pi, t)] $.
This is an equality by Proposition~\ref{prop: Langlands class}.
\end{proof}

For  an $\He$-module $V$, the $\Ab$-character of $V$ is the element of $\C[T]$
\[
\Theta_\Ab(V) = \sum_{s \in \Omega(V)} [s : V] \, s .
\]

\begin{theorem} \label{Thm: lin indep}
Suppose $\He $ is quasi-simply-connected. The  $\Ab$-characters of simple $\He$-modules are $\C$-linearly independent.    
\end{theorem}

\begin{proof}
Let
\begin{equation} \label{eq: weights}
    \sum_i c_i \Theta_{\Ab}(L_i) = \sum_i  \sum_{s \in \Omega(L_i)} c_i [s : L_i] s = 0
\end{equation}
be a  finite linear combination of $\Ab$-characters with the $L_i$ distinct simple $\He$-modules. Take $s \in \cup_i \Omega(L_i)$ such that $\lambda = \log |s|$ is maximal. Then  $\lambda^0$ is maximal among $\{ (\lambda')^0 : \lambda' \in \cup_i \Lambda(L_i) \}$ \cite[(2.13)]{RamKril}. 
Let $P = F_\lambda$ and $\nu = \lambda^0$ as in (\ref{eq: Langlands}). Fix a $t \in T^P$ with $\log|t| = \nu$ occurring among the Langlands parameters for the $L_i$. Let  $\pi_1, \dots, \pi_k$ be a complete list of distinct irreducible tempered $\He_P$-modules such that $  L(P, \pi_l, t)$  occurs in (\ref{eq: weights}).
By Lemma~\ref{lemma: t=t'} this also accounts for any $L(P, \pi', t')$ with $\log|t'| = \nu$. 
We  show the coefficients of $s_j \otimes t$ for  $s_j \in \Omega(\pi_l)$ come only from $\cup_{l} L(P, \pi_l, t)$.

 Let $\lambda_j = \log|s_j|$. Then as $\pi_l$ is tempered it follows that $(\lambda_j + \nu)^0 = \nu$ and $F_{\lambda_j + \nu} = P$. Suppose $\lambda_j + \nu \in \Lambda(L(P' , \pi', \nu'))$. Take $\lambda' \in \Lambda(L(P' , \pi', t'))$ maximal so $\lambda' \geq \lambda_j + \nu$. By the Langlands classification $P' = P_{\lambda'}$ and $\log|t'| = (\lambda')^0$. Then $(\lambda')^0 \geq (\lambda_j + \nu)^0 = \nu$ hence by maximality $(\lambda')^0 = \nu = \lambda^0$ and  $ P_{\lambda'} = P$.
Hence the coefficients of $s_j \otimes t$ in (\ref{eq: weights}) come only from  $\cup_{l} L(P, \pi_l, t)$. 
By Lemma~\ref{lemma: weights mult.}
\[
\sum_l c_l \Theta_{\Ab_P}(\pi_l) = \sum_l \sum_{s_j \in \Omega(\pi_l)} c_l [s_j  : \pi_l]{s_j} = \sum_l \sum_{s_j \in \Omega(\pi_l)} c_l [s_j \otimes t : L(P, \pi_l , t]({s_j \otimes t})= 0.
\] 

If $P \neq \Delta$ we are done by induction. If $P = \Delta$ then the combination involves only $\Ab$-characters of irreducible tempered $\He$-modules. Apply the Iwahori-Matsumoto involution $\IM$. This preserves irreducibility and $\IM(\pi_s) = (\IM(\pi))_{s^{-1}}$ so 
repeating the same argument we may assume then that all $\pi_l$ are tempered and anti-tempered. 
Furthermore, a simple $\He$-module has all its weights in $\cc(V) = W\cdot s$ and by partitioning different orbits we may suppose all weights in (\ref{eq: weights}) are in the same orbit; equivalently all simple modules $L_i$ have the same central character. By Theorem~\ref{thm: EM} there is a unique such $L_i$.
\end{proof}

\begin{remark} 
    Theorem \ref{Thm: lin indep} has also  been found by Antor and Okada \cite[Corollary 3.15]{AntorOkada} by 
    first showing this property holds for graded Hecke algebras and then 
    lifting to quasi-simply-connected affine Hecke algebras.  
\end{remark}

\begin{theorem}
    Let $\He$ be quasi-simply-connected and suppose the irreducible tempered $\He_P$-modules have been constructed for all $P \subset \Delta$. Then there is a finite algorithm for determining the Jordan-H\"older factors with multiplicities of constructed finite length $\He$-modules.
\end{theorem}
\begin{proof}
   Fix a constructed finite length $\He$-module $V$. Diagonalising the action of the $\theta_{x_i}$ gives the character $\Theta(V)$. By first partitioning the set of weights  into $W$-orbits we may assume $V$ has highest weight $s$ and central character $W \cdot s$.  Let $\lambda = \log|s|$ and compute $P = P_\lambda \subset \Delta$ and hence $\nu = \log|s|^0$ using the property (\ref{eq: Langlands}) (recall we assume there are finite algorithms for handling root data, so one may compute $P$ by enumerating all subsets of $F_P \subset \{ 1, \dots ,n\}$ and for each one checking if $\lambda$ is expressed appropriately in the basis $P^\vee \cup \{\omega_j^\vee : j \not\in F_P  \}$). 

   Let $K_1  = X \cap P^{\vee\perp}$ and $K_2 = X \cap \Q P$. These intersect trivially so let $t \colon K_1 + K_2 \to \C^\times , k_1 + k_2 \mapsto s(k_1)$. As
   $\C^\times$ is injective this extends to a map $t \colon X \to \C^\times$ with $\log |t|  = \nu$. Furthermore $r = s t^{-1}$ descends to $X / K_1 = X_P$.
    
    There is a finite set of irreducible tempered $\He_P$-modules $\pi_1, \dots, \pi_k$ with central character $W_P \cdot r$. 
    In each case construct $L(P, \pi_l, t)$ using Proposition \ref{prop: algorithm} and hence $\Theta(L(P, \pi_l, t))$. 
    By Lemma \ref{lemma: t=t'} this exhausts all simple $\He$-modules associated to $\nu$.

     By Theorem~\ref{Thm: lin indep} there is a unique $\C$-linear combination of the $\Theta(L(P, \pi_l, t))$ matching the $s_j \otimes t$, $s_j \in \bigcup \Omega(\pi_l)$ part of  $\Theta(V)$. This gives a set of linear equations with unique solution giving the multiplicities for each Jordan-H\"older factor $L(P, \pi_l, t)$ in $V$.
    Subtract these from $\Theta(V)$ and repeat.
\end{proof}


    


\subsection{Graded affine Hecke algebras}
Let $\gh = \gh(R, \afr^* , k)$ be the graded affine Hecke algebra associated to a root system $(R, \afr^*)$ with parameters $k \colon R \to \Re \backslash \{ 0 \}$. This has a large commutative subalgebra $\ga$ isomorphic to the symmetric algebra $S(\afr^*)$.
A preprint of Barbasch and Ciubotaru \cite{BC2013} outlines a proof of linear independence of $\ga$-characters of simple $\gh$-modules which has been adapted  to prove Theorem~\ref{Thm: lin indep}. The preceding results also yield a complete proof of their original result.

\begin{theorem}[{\cite[Theorem 4.4.2]{BC2013}, see also \cite[Proposition 3.6]{AntorOkada}}] 
     The   $\ga$-characters of simple $\gh$-modules are linearly independent. 
\end{theorem}

\begin{remark}
    The trivial and sign modules of the degenerate algebra  $\gh(R, \afr^*, 0) = \C[W] \ltimes S[\afr^*] $  have identical $\ga$-character.
\end{remark}

\section{Isogenies and central morphisms}

\begin{definition} \label{def: centralMorphism}
    A \emph{central morphism} of based root data $ (X_1,R_1,Y_1,R_1^\vee , \Delta_1) \to (X_2,R_2,Y_2,R_2^\vee , \Delta_2)$ is an injective group homomorphism $f \colon X_1 \to X_2$ 
   such that the restriction of $f$ to  $ R_1$   defines an isomorphism of based root systems $(\Re R_1 , R_1 , \Re R_1^\vee , R_1^\vee , \Delta_1) \to (\Re R_2 , R_2 , \Re R_2^\vee , R_2^\vee , \Delta_2)$.
    
    An \emph{isogeny} of based root data is a central morphism such that $f \colon X_1 \to X_2$ has finite cokernel.
\end{definition}


\begin{definition}
    Let $ \R_1 \to \R_2$ be a central morphism. There is a corresponding
  homomorphism $ \He_1 \to \He_2$ of affine Hecke algebras provided the parameters are compatible, which we also call a \emph{central morphism}, sending $\theta_x T_w \mapsto \theta_{f(x)} T_{w}$.  Write $\Gamma_f = X_2 / f(X_1)$.  For $g \in \Gamma_f$ let $\He_{2,g} = \C \{ \theta_x T_w : x + f(X_1)  = g, w \in W \}$. 
\end{definition}
As $f$ is injective on $X$ any  central morphism of affine Hecke algebras is an inclusion and we shall identify $X_1 \subset X_2$ and $\He_1 \subset \He_2$.  

Isogenies are used in \cite{Reeder2002} to determine Langlands parameters for simple modules of certain affine Hecke algebras with equal parameters by reducing to the simply-connected case, using Clifford theory of \cite{RamRam2003}. This approach applies also to affine Hecke algebras with not equal parameters (see Remark \ref{rmk: RR}). 

Induction and restriction along isogenies for semisimple affine Hecke algebras is studied in \cite[Lemma 6.3]{DelormeOpdam}. Delorme and Opdam present isogenies as  extensions by finite groups of length zero elements. This  occurs naturally from the Iwahori-Matsumoto presentation of an affine Hecke algebra and is implicit in the proof of Theorem \ref{thm: signature} however we shall not use this directly.

\begin{lemma}
    Let $f \colon \He_1 \to \He_2$ be a central morphism. Then $\theta_x \He_{2,g} \theta_{-x} = \He_{2,g}$ for any $x \in X_2$. In particular $\theta_x \He_1 \theta_{-x} = \He_1$.
\end{lemma}
\begin{proof}
   It suffices to show $T_s \theta_x \in \He_{2, x + X_1} $ for all $s \in W$.  By \cite[Proposition 3.6]{Lusztig}, if $\alpha^\vee \not\in 2Y_2$ then
    \begin{equation*} \label{eq: graded alg}
        T_s \theta_x  = \theta_{s(x)} T_s + 
        (q^{\lambda(\alpha)} - 1) \frac{\theta_x - \theta_{s(x)}}{1 - \theta_{- \alpha}} = \theta_{s(x)} T_s + 
        (q^{\lambda(\alpha)} - 1)  \theta_x \frac{1 - \theta_{-\alpha}^{\la x, \alpha^\vee\ra}}{1 - \theta_{- \alpha}}.
    \end{equation*}
   As $W$ acts trivially on $X_2 / \Z R$ and $\Z R \subset X_1$ it follows 
    that $T_s \theta_x \in \He_{2, x + X_1} $.
    If $\alpha^\vee \in 2Y_2$ then
    \begin{align*} \label{eq: graded alg adj}
        T_s \theta_x  &= \theta_{s(x)} T_s + 
        \left((q^{\lambda(\alpha)} - 1) + \theta_{- \alpha}\left(q^{\frac{\lambda(\alpha) + \lambda^*(\alpha)}{2}} - q^{\frac{\lambda(\alpha) - \lambda^*(\alpha)}{2}}\right) \right) \frac{\theta_x - \theta_{s(x)}}{1 - \theta_{- 2\alpha}}  
    \end{align*}
    and the same argument applies.
\end{proof}

\begin{theorem} \label{Thm: ResInd}
Let $f \colon \He_1 \to \He_2$ be a central morphism. For an $\He_1$-module   $\pi$  let $\pi^g(h) = \pi(\theta_x h \theta_{-x})$ where $g = x + X_1$. Then
\begin{enumerate}
    \item $\Res{\He_1}{\He_2} \Ind{\He_1}{\He_2} \pi = \bigoplus_{g \in \Gamma_f} \pi^g$
   with  all $\pi^g$ independent of the choice of $x$ up to isomorphism and $\Res{\Ab_1}{\He_1}\pi^g = \Res{\Ab_1}{\He_1}\pi$. Furthermore if $\pi$ is irreducible  then so is each $\pi^g$ and if $\pi$ is semisimple then so is  $\Res{\He_1}{\He_2} \Ind{\He_1}{\He_2} \pi$, and
   \item if $f \colon \He_1 \to \He_2$ is an isogeny  and $\pi$ is semisimple then so is  $\Ind{\He_1}{\He_2} \pi$.
\end{enumerate}
\end{theorem}
\begin{proof}
  Consider the linear decomposition \[\Ind{\He_1}{\He_2} \pi = \He_2 \otimes_{\He_1} \pi = \C[\Gamma_f] \otimes_{\C} \pi = \bigoplus_{g \in \Gamma_f} \theta_x \otimes \pi = \bigoplus_{g \in \Gamma_f} \pi^g.\] 
By the previous lemma $\theta_x \otimes \pi$ is  the twist of $\pi$ by the automorphism induced by $\theta_x$.  The isomorphism class of the resulting module depends only on the outer automorphism induced by $\theta_x$ and hence the $\He_1$-module $\pi^g = \theta_x \otimes \pi$ is independent of the choice of $x$.  As $\Ab_1 \subset \Ab_2$ is commutative, $\pi^g$ is isomorphic to $\pi$ when restricted to $\Ab_1$.  Furthermore  twisting by an automorphism of $\He_1$ preserves irreducibility and semisimplicity, hence $\Res{\He_1}{\He_2} \Ind{\He_1}{\He_2} \pi$ is semisimple.
  
For the second claim let $U$ be an $\He_2$-submodule of $\Ind{\He_1}{\He_2} \pi$. 
By the first claim there exists an $\He_1$-equivariant retraction $p \colon \Ind{\He_1}{\He_2} \pi \to U$. Let 
   \[
   \Tilde{p}(m) = \frac{1}{n} \sum_x \theta_{x} p (\theta_{-x} m) \qquad m \in \Ind{\He_1}{\He_2} V.
   \]
   As $U$ is stable under $\theta_x \in \He_2$ it follows that $\Tilde{p}(u) = u$ for $u \in U$.
   It remains to show $\Tilde{p}$ is $\He_2$-equivariant. Clearly $\Tilde{p}(\theta_x m) = \theta_x \Tilde{p}(m)$ and as conjugation by $\theta_x$ is an automorphism of $\He_1$ we have
   \[
   \Tilde{p}(h m) = \frac{1}{n} \sum_x \theta_{x} p ((\theta_{-x} h \theta_x) \theta_{-x} m)
   =\frac{1}{n} \sum_x \theta_{x} (\theta_{-x} h \theta_x) p ( \theta_{-x} m) = h \Tilde{p}(m) .
   \]
\end{proof}

\begin{corollary} \label{cor: free}
    Let $\He_1 \to \He_2$ be a central morphism. There is an $\Gamma_f$-grading  $\He_2 \cong \bigoplus_{g \in \Gamma_f} \He_{2,g}$ with $\He_{2,g}  \cong \He_1$ and $\He_{2,g} \He_{2, g'} = \He_{2,g + g'}$. Each graded component is free of rank one as a left and a right $\He_1$-module. The linear decomposition $\Ind{\He_1}{\He_2} \pi = \bigoplus_{g \in \Gamma_f} \pi^g$ makes $\Ind{\He_1}{\He_2}V$ a graded module for the graded algebra $\He_2$.
\end{corollary}

\begin{remark} \label{rmk: RR}
    It is shown in \cite{RamRam2003} that 
    if $f \colon \He_1 \to \He_2$ is an isogeny then $\He_2$ acquires an action of $\Gamma_f$ for which $\He_1$ is the set of fixed points. In our setting this is made apparent by decomposing $\He_2 =\bigoplus_{g \in \Gamma_f} \He_{2,g}$ and having $g \in X_2 / X_1$ scale each block by a suitable root of unity. 
\end{remark}    

The following is well known.
\begin{lemma} \label{lemma: existence of SC1}
    Let $\R = (X, R, Y, R^\vee, \Delta)$ be a based root datum. There exists a simply-connected root datum $\R^\scon$ and an isogeny $\R \to \R^\scon$.   
\end{lemma}
\begin{proof}
    Let $Y' = Y \cap (R \otimes \Q)^\perp \oplus \Z R^\vee$ and $X' = \Hom_{\Z} (Y' , \Z)$. 
    The inclusion $Y' \to Y$ gives a restriction map $f \colon \Hom_\Z(Y, \Z) = X \to X'$.  This induces the identity upon extending scalars to $\Re$, hence $f$ has finite kernel and cokernel. The only finite subgroup of $X$ is trivial so $f$ is injective with finite cokernel. 
    Identify $R$ with its image under $f$.
    Then $\R^\scon = ( X', R , Y' , R^\vee , \Delta)$ is simply-connected and $\R \to \R^\scon$ is an isogeny.
\end{proof}

\begin{lemma} \label{lemma: existence of SC}
    Let $\He(\R)$ be an affine Hecke algebra with possibly unequal parameters. Then there exists a quasi-simply-connected affine Hecke algebra $\He^\scon$ and an isogeny $\He(\R) \to \He^\scon$. 
\end{lemma}
\begin{proof}
Recall we may suppose $\R$ is primitive. The only obstruction to the isogeny of affine Hecke algebras is from the parameters, which must remain constant on orbits of $\Delta$ under the extended affine Weyl group of the quasi-simply-connected affine Hecke algebra. If $\R = \R^\scon$ there is nothing to prove
and if $\R = \R_s$ then the conjugacy classes are independent of the isogeny class. This leaves the case $\R=\R_B$   adjoint of type $B$. The parameters may cause an obstruction only if $\lambda(\alpha) \neq \lambda^*(\alpha)$ for the short simple root, in which case $\He(\R)$ is by definition quasi-simply-connected.
\end{proof}

Next we give an analogue of Lemma~\ref{inducedWeights} for induction along isogenies.

\begin{lemma} \label{lemma: inducedWeights} 
   Let $\He_1 \to \He_2$ be an isogeny. The $\Ab_2$-weights of $ \Ind{\He_1}{\He_2} \pi$ are
\[
\Omega_{\Ab_2}( \Ind{\He_1}{\He_2} \pi) = \{ s \colon X_2 \to \C^* : s|_{X_1} \in \Omega(\pi) \} 
\]
and the $\Ab_2$-character  is
\[
\Theta_{\Ab_2} (\Ind{\He_1}{\He_2} \pi) = \sum_{s' \in \Omega_{\Ab_1}(\pi)} \sum_{ s|_{X_1} = s' } [s |_{X_1} : \pi] s .
\]
\end{lemma}
\begin{proof}
   Follows from
\[
\Theta_{\Ab_2}(\Ind{\He_1}{\He_2} \pi) = \Theta_{\Ab_2}(\Res{\Ab_2}{\He_2}\Ind{\He_1}{\He_2} \pi) = \Theta_{\Ab_2}(\Ind{\Ab_1}{\Ab_2} \Res{\Ab_1}{\He_1} \pi)
\]
so $\Ind{\He_1}{\He_2}$ and $\Res{\He_1}{\He_2}$ coincide with 
 usual induction and restriction on character rings.
\end{proof}

 \begin{corollary} \label{cor: inducedCharacters}
  Let $\He_1 \to \He_2$ be an isogeny. Let $T_1 = \Hom(X_1, \C^\times)$ and $T_2 = \Hom(X_2, \C^\times)$. Define 
 \begin{align*}
 \Ind{\He_1}{\He_2} \colon \C[T_1] \to \C[T_2], \: \: &s' \mapsto \sum_{ s|_{X_1} = s' }  s \\   
 \Res{\He_1}{\He_2} \colon \C[T_2] \to \C[T_1], \: \: &s  \mapsto s|_{X_1}. 
 \end{align*}
  These are linear maps satisfying
$\Theta_{\Ab_2} (\Ind{\He_1}{\He_2} \pi) = \Ind{\He_1}{\He_2} (\Theta_{\Ab_1} ( \pi)) $ and $
   \Theta_{\Ab_1} (\Res{\He_1}{\He_2} \pi) = \Res{\He_1}{\He_2} (\Theta_{\Ab_2} ( \pi))   $. 
 \end{corollary} 

\begin{corollary}[{\cite[Lemma 6.3]{DelormeOpdam}}] \label{cor: preservation properties}
    Induction and restriction along isogenies preserve the properties of being  tempered, antitempered, and discrete series.
\end{corollary}

Let $\He \to \He^\scon$ be the isogeny given by Lemma~\ref{lemma: existence of SC}.  The set $\Irr(\He)$ decomposes into conjugacy class packets $\Pi(V) = \{ V^g  : g \in X^\scon / X \}$ where $X^\scon$ is the character lattice of the isogenous quasi-simply-connected affine Hecke algebra. The number of isomorphism classes occurring in $\Pi(V)$ clearly divides $|X^\scon / X|$, and is not independent of $V$ (see \S4.1).

\begin{theorem}
    Let $V \in \Irr \He$. Then $\Omega(V)$ uniquely determines $\Pi(V)$.
\end{theorem}
\begin{proof}
    Suppose $\Omega(V) = \Omega(V')$ for some $V' \in \Irr{\He}$, and that the result is known for all proper parabolics. If $V$ is not tempered then write $V = L(P , \pi , t)$ and $V' = L(P' , \pi', t')$. Then as in \cite[5.5 Theorem]{EM} 
    we have $P = P'$, $t = t'$, and $\Omega(\pi) = \Omega(\pi')$ so by induction $\pi' = \pi^g$. There is an isomorphism     $M(P, \pi^g, t) \to M(P, \pi,t)^g$
given by $h \otimes v \mapsto \theta_g h \theta_{g^{-1}} \otimes v$ and hence also on the unique  quotient $L(P, \pi^g, t) \cong L(P, \pi, t)^g$ so $V' = V^g$.
If $V$ is tempered but not antitempered then apply the $\IM$ and repeat the same argument.

    This leaves the case that $P = \Delta$ and $V$ is tempered and antitempered.  Then $\Ind{\He}{\He^\scon}V$ and $\Ind{\He}{\He^\scon}V'$ are also tempered and antitempered. Let  $s \in \Omega(\Ind{\He}{\He^\scon}V) = \Omega(\Ind{\He}{\He^\scon}V')$.  By Theorem~\ref{Thm: ResInd} this is semisimple so $s$ occurs as a weight of a simple submodule $L$. Similarly there is a simple submodule $L'$ of $\Ind{\He}{\He^\scon}V'$ with weight $s$. By Corollary~\ref{cor: preservation properties}, $L$ and $L'$ are tempered and antitempered. As they share a weight they also share central character so $L \cong L'$ by Theorem~\ref{thm: EM}. The result now follows from the decompositions $\Res{\He}{\He^\scon} L \subset \Res{\He}{\He^\scon} \Ind{\He}{\He^\scon}V = \bigoplus_{g \in X^\scon / X} V^g$  and $\Res{\He}{\He^\scon} L' \subset \bigoplus_{g \in X^\scon / X} V'^g$ given by Theorem~\ref{Thm: ResInd}. 
\end{proof}

 \begin{theorem} \label{thm: lin indep classes}
     The $\Ab$-characters of simple $\He$-modules in distinct conjugacy classes are linearly independent. 
 \end{theorem}
 \begin{proof}
    Let
\begin{equation} \label{eq: weightsNonSC}
    \sum_i c_i \Theta_{\Ab}(L_i) =  0
\end{equation}
be a  finite linear combination of $\Ab$-characters with the $L_i$ simple $\He$-modules in distinct conjugacy classes.
By Corollary~\ref{cor: inducedCharacters},
\[
\sum_i c_i \Theta_{\Ab^\scon}(\Ind{\He}{\He^\scon}L_i) =  \Ind{\He}{\He^\scon} ( \sum_i c_i  \Theta_{\Ab}(L_i) ) = 0.
\]

Let $\pi_i^j$ be the composition factors of  $ \Ind{\He}{\He^\scon} L_i$ and write $m_{i,j}$ for the multiplicity. Then
\[
\Theta_{\Ab^\scon}(\Ind{\He}{\He^\scon}L_i)  = \sum_j m_{i,j} \Theta_{\Ab^\scon}(\pi_i^j).
\]
and substituting this into (\ref{eq: weightsNonSC}) gives
\[
\sum_i \sum_j c_i  m_{i,j} \Theta_{\Ab^\scon}(\pi_i^j) = 0.
\]

Suppose $\Ind{\He}{\He^\scon}L_i$ and $\Ind{\He}{\He^\scon}L_k$ have a common summand $\pi_i^j$.
Then the same holds for $\Res{\He}{\He^\scon} \Ind{\He}{\He^\scon}L_i$ and $\Res{\He}{\He^\scon} \Ind{\He}{\He^\scon}L_k$.
By Theorem~\ref{Thm: ResInd} this cannot hold for $L_i, L_k$ in distinct conjugacy classes.
 Hence each $\Theta_{\Ab^\scon}(\pi_i^j)$ is distinct and occurs with coefficient  $c_i m_{i,j}$. By Theorem~\ref{Thm: lin indep}  the $c_i$ are all zero.
 \end{proof}

\subsection{Rank 1 example} \label{sec: example}
    The Iwahori-Hecke algebra $\He(\PGL_2, q ) = \He_I(\SL_2(k))$ has equal parameters $q = | \mathcal{O} / \varpi \mathcal{O} | > 1$ and associated $\PGL_2$ root datum $X = \Z x, R = \{ \pm  x\}, Y = \Z x^\vee, R^\vee = \{ \pm 2 x^\vee \} , \Delta = \{x\}$. 
    The finite Weyl group is $W = \{1, s\}$ and the Bernstein presentation is $\He(W) \otimes \C[\Z]$ with cross relation
    \begin{equation} \label{eq: crossPGL2}
        \theta_x T_s = T_s \theta_{-x} + (q-1)(\theta_x + 1) .
    \end{equation}
    The minimal principal series representations $\pi_\tau = \Ind{\Ab}{\He} \tau$ have basis $(T_s + 1) \otimes 1, (T_s - q) \otimes 1$ with respect to which
    \[\pi_{\tau} : T_s \mapsto \begin{pmatrix}
  q &   \\
   & -1
\end{pmatrix}\] and
\begin{align*}
        \pi_{\tau} :  \theta_x \mapsto \frac{1}{1 + q}\begin{pmatrix}
  \tau q +  \tau^{-1} q + q - 1 & 
  \tau^{-1} q -  \tau + q - 1\\
   \tau^{-1}  -  \tau q - q + 1 & 
   \tau  +  \tau^{-1} - q + 1
\end{pmatrix}.
    \end{align*} 
    This has $\Ab$-character $\tau + { \tau^{-1}}$ and is irreducible unless $\tau \in \{q , -1, q^{-1} \}$. 
When $\tau \in \{ q ,  q^{-1} \}$, it is not semisimple and the subquotients are the trivial   and Steinberg representations \[\triv \colon T_s \mapsto q , \; \theta_x \mapsto q \qquad \St \colon T_s \mapsto -1 , \; \theta_x \mapsto q^{-1} \] which are $q$-deformations of the trivial and sign representations respectively. These fit exact sequences
\begin{align*}
    0 \to \St \to &\pi_q \to \triv \to 0 \\
    0 \to \triv \to &\pi_{q^{-1}} \to \St \to 0.
\end{align*}
When $\tau = -1$ it splits with summands 
\[X^- \colon T_s \mapsto q , \; \theta_x \mapsto -1 \qquad X^+ \colon T_s \mapsto -1 , \; \theta_x \mapsto -1. \]

The Iwahori-Hecke algebra $\He(\SL_2, q) = \He_I(\PGL_2(k))$ also has equal parameters $q>1$ and associated simply-connected $\SL_2$ root datum $X = \Z x^\vee, R = \{ \pm  2x^\vee \}, Y = \Z x, R^\vee = \{ \pm  x \} , \Delta = \{2x^\vee\}$. The cross relation is now
    \begin{equation} \label{eq: crossSL2}
        \theta_{x^\vee} T_s = T_s \theta_{-x^\vee} + (q-1)\theta_{x^\vee}  .
    \end{equation}
    The minimal principal series representations $\sigma_\tau = \Ind{\Ab}{\He} \tau$ have basis $(T_s + 1) \otimes 1, (T_s - q) \otimes 1$ with respect to which
    \[\sigma_{\tau} : T_s \mapsto \begin{pmatrix}
  q &   \\
   & -1
\end{pmatrix}\] and
\begin{align*}
        \sigma_{\tau} :  \theta_{x^\vee} \mapsto \frac{1}{1 + q}\begin{pmatrix}
  \tau q +  \tau^{-1} q  & 
  \tau^{-1} q -  \tau \\
   \tau^{-1}  -  \tau q  & 
   \tau  +  \tau^{-1} 
\end{pmatrix}.
    \end{align*} This has $\Ab$-character $\tau + { \tau^{-1}}$
    and is irreducible unless $\tau \in \{\pm q^{1/2} , \pm q^{-1/2} \}$. For these $\tau$ it is not semisimple.
When $\tau \in \{   q^{ 1/2} , q^{ -1/2}  \} $ the subquotients are  \[\triv_+ \colon T_s \mapsto q , \; \theta_{x^\vee} \mapsto  q^{1/2} \qquad \St_+ \colon T_s \mapsto -1 , \; \theta_{x^\vee} \mapsto  q^{-1/2} \] and when $\tau \in \{   -q^{ 1/2} , - q^{ -1/2}  \}$ the subquotients are 
\[\triv_- \colon T_s \mapsto q , \; \theta_{x^\vee} \mapsto  -q^{1/2} \qquad \St_- \colon T_s \mapsto -1 , \; \theta_{x^\vee} \mapsto  -q^{-1/2}. \]

The isogeny $\SL_2 \to \PGL_2, x \mapsto 2x^\vee$ gives an algebra inclusion $\He(\PGL_2) \to \He(\SL_2)$ mapping $\theta_x \mapsto \theta_{2 x^\vee}$; indeed from equation (\ref{eq: crossSL2}) one recovers equation (\ref{eq: crossPGL2})
\begin{equation*}
    \theta_{2 x^\vee} T_s = \theta_{ x^\vee} (T_s \theta_{-x^\vee} + (q-1)\theta_{x^\vee}) = T_s \theta_{2 x^\vee}^{-1} + (q-1)(\theta_{2 x^\vee} + 1)
\end{equation*}
and the graded algebra structure is
\begin{equation*}
    \He(\SL_2) = \He(\PGL_2) \oplus \theta_{x^\vee} \He(\PGL_2).
\end{equation*}
Induction and restriction with respect to this inclusion gives
\[
\begin{array}{l c l l l l l l l}
 \Ind{}{} \: \pi_\tau &= & \sigma_{\tau^{1/2}} \oplus \sigma_{-\tau^{1/2}} &  &  \Res{}{} \: \Ind{}{} \: \pi_\tau &= & \pi_\tau \oplus \pi_\tau  \\
   \Ind{}{} \: \triv &= & \triv_+ \oplus \triv_- &  &  \Res{}{} \: \Ind{}{} \: \triv &= & \triv \oplus \triv  \\
    \Ind{}{} \:  \St &= & \St_+ \oplus \St_- & &  \Res{}{} \:  \Ind{}{} \:  \St &= & \St \oplus \St & & \\
    \Ind{}{} \:  X^- &= & \Ind{}{} \:  X^+ = \sigma_i  &&  \Res{}{} \:  \Ind{}{} \:  X^- &= & \Res{}{} \: \Ind{}{} \:  X^+ = X^- \oplus X^+.
\end{array}
\]
Observe  $X^+ = (X^{-})^{ {x}^\vee}$. 

\section{Hermitian duals and unitarity} 
Affine Hecke algebras have a natural star operation (conjugate-linear anti-automorphism) $*$ defined on the Iwahori-Matsumoto basis by $T_w^* = T_{w^{-1}}$. On the commutative subalgebra this is $ \theta_x \mapsto T_{w_0} \theta_{-w_0(x)} T_{w_0}^{-1} $ where $w_0$ is the long element of the finite Weyl group.
A proper parabolic subalgebra $\He^P \to \He$ will generally not be a $*$-subalgebra, however injective central morphisms $\He_1 \to \He_2$ will exhibit $\He_1$ as a $*$-subalgebra of $\He_2$. 

A module $V$ is hermitian if it may be  equipped with a nondegenerate hermitian form $\la, \ra$  satisfying $*$-invariance:
\[
\la h v , v' \ra = \la v , h^* v' \ra.
\]
This induces a map $V \to V^\dagger =  \Hom_\C (\overline{V} , \C) $ where $\overline{V}$ is the complex conjugate representation and $\He$ acts on $V^\dagger$ by $(h \lambda)(v) = \lambda (h^* v) $ for $h \in \He, \lambda \in V^\dagger, v \in V$.
 A module is unitary if it may be equipped with a positive definite $*$-invariant hermitian form.

\begin{lemma}
    Let $\He$ be quasi-simply-connected and $V$ a simple $\He$-module. Then $V$ admits a nondegenerate $*$-invariant hermitian form 
    if and only if $\Omega(V)$ is closed under $s \mapsto \overline{w_0(s^{-1})}$.
\end{lemma}
\begin{proof} 
   By  Theorem~\ref{thm: weights} it suffices to show  $\Omega(V^\dagger) = \{ \overline{w_0(s^{-1})} : s \in \Omega(V) \}$. This is \cite[Lemma 4.1]{Solleveld_Opdam_2026} (indeed it is clear from  $\theta_x^* = T_{w_0} \theta_{-w_0(x)} T_{w_0}^{-1} $ being conjugate to $\theta_{-w_0(x)}$).
\end{proof}

\begin{example}
    We continue with the notation in \S\ref{sec: example}.  The  simple $\He(\SL_2)$-modules  with weights fixed under $s \mapsto \overline{w_0(s^{-1})}$ are 
    $\sigma_\tau$ for $\tau \in \Re^* \backslash \{\pm q^{1/2} , \pm q^{-1/2} \} \cup \{z \in \C : |z| = 1 \}$ and all the one-dimensional modules.
    Unfortunately there does not appear to be a similar criterion on weights for detecting unitary modules, which in this case are the one-dimensional modules, $\sigma_\tau$ for  $\tau \in \{z \in \C : |z| = 1 \} $ (unitary principal series) and $\tau \in \Re$ such that $q^{-1/2} < | \tau | < q^{1/2}$ (complementary unitary series). 
    %
\end{example}

\begin{proposition}
    Let $\He_1 \to \He_2$ be a central morphism
    and $V$ an $\He_2$-module. Then $\Res{\He_1}{\He_2} (V^\dagger) = (\Res{\He_1}{\He_2} V)^\dagger$. If $V$ is unitary then so is $\Res{\He_1}{\He_2}V$. If $V$ is unitary and finite-dimensional  then $\Res{\He_1}{\He_2}V$ is a direct sum of simple unitary modules.
\end{proposition}
\begin{proof}
    The  central morphism exhibits $f(\He_1)$ as a $*$-subalgebra of $\He_2$, and finite-dimensional unitary modules are semisimple.
\end{proof}

\begin{lemma} \label{lemma: stargraded}
    The star operation $*$ makes $\He_2$ a graded $*$-algebra, that is, $\He_{2 , g}^* = \He_{2, -g}$.
\end{lemma}
\begin{proof}
    This follows from $\theta_x^* = T_{w_0} \theta_{-w_0(x)} T_{w_0}^{-1}$ and $w_0$ acting trivially on $X_2 / X_1$.
\end{proof}

\begin{theorem} \label{theorem: induced pairing}
   Let $\He_1 \to \He_2$ be an isogeny and let $E \colon \He_2 \to \He_1$ be the projection onto $\He_{2, 0} = \He_1$.  The pairing
    \[
    \begin{array}{c c c }
        \Ind{\He_1}{\He_2} (V^\dagger) \times \Ind{\He_1}{\He_2} V &\to &\C \\
        \la h' \otimes \lambda , h \otimes v \ra & =  &(E(h^* h')\lambda)(v)
    \end{array}
    \]
    induces an isomorphism $ \Ind{\He_1}{\He_2} (V^\dagger) \cong (\Ind{\He_1}{\He_2} V)^\dagger $. 
\end{theorem}

\begin{proof}
The pairing comes from the composition of the following maps
\[
\begin{array}{c c  c  c c}
    \He_2 \otimes_{\He_1} V^\dagger & \to  &\Hom_{\He_1}(\He_2 , V^\dagger ) &\to &\Hom_{\C}(\He_2 \otimes_{\He_1} \overline{V} , \C) \\
    h' \otimes \lambda &\mapsto &(h \mapsto E(h^* h') \lambda) & \mapsto & (h \otimes v \mapsto (E(h^* h') \lambda)(v)).
\end{array}
  \]
  where $\Hom_{\He_1}(\He_2 , V^\dagger )$ is the space of morphisms $\He_2 \to V^\dagger$ satisfying $f(h_2 h_1) = h_1^* f(h_2)$.
The first map is well-defined on $\He_2 \otimes_{\He_1} V^\dagger$ as $E(h' h_1) = E(h')h_1$ for $h_1 \in \He_1$. By construction the pairing is $*$-invariant for $\He_2$ so the maps are $\He_2$-equivariant.

 As $\He_1 \to \He_2$ is an isogeny,  $\He_2$ is a finitely generated projective $\He_1$-module by Corollary~\ref{cor: free}  so the first map is an isomorphism for all $V$.
  The second map is an isomorphism by the tensor-hom adjunction.
\end{proof}

\begin{theorem}[{\cite[Lemma 6.3]{DelormeOpdam}}] \label{thm: signature}
 Let $\He_1 \to \He_2$ be an isogeny with  $|X_2 / X_1| = n$ and $V$ a finite dimensional $\He_1$-module with $*$-invariant hermitian form $\la, \ra$ with signature $(p,q)$. Then $\Ind{\He_1}{\He_2} V$ admits a $*$-invariant hermitian form with signature $(np, nq)$. In particular  $\Ind{\He_1}{\He_2} $ preserves unitarity.
\end{theorem}
\begin{proof}
Recall the $X_2 / X_1$ grading $\Ind{\He_2}{\He_1} V = \bigoplus_g  V^g$ where $V^g = \He_{2, g} \otimes_{\He_1} V = \theta_g \otimes_\C V$. By Lemma~\ref{lemma: stargraded} the pairing in Theorem~\ref{theorem: induced pairing} satisfies $\la V^g, V^{g'} \ra = 0$ unless $g = g'$, so it suffices to consider the signature on the restriction to $V^g \times V^g$.

    We refine the choice of representative for $\He_{2,g} = \theta_x \He_1$ where $x + X_1 = g$.
  Write $x = x^+ - x^-$ with $x^+, x^- \in X_2$ dominant, chosen such that $x^- \in X_1$.  
   Then $\theta_x \He_1= T_{x^+} \He_1$.  
   As $W \ltimes \Z R $ acts simply transitively on the set of alcoves, there is a $w \in W \ltimes \Z R$ such that ${x^+}w$ preserves the fundamental alcove and so $l({x^+}w)=0$, and $l({x^+}) = l(w) = l(w^{-1})$ so $T_{{x^+}w}T_{w^{-1}} = T_{{x^+} w w^{-1}} = T_{{x^+}}$. Hence $\theta_x \He_1 = T_{x^+} \He_1 = T_{{x^+}w} \He_1$ and furthermore $(T_{{x^+}w})^* T_{{x^+}w} = T_{({x^+}w)^{-1}} T_{{x^+}w} = T_1$.

Let $T_{{x^+}w} \otimes v \in V^g$. Then $\la T_{{x^+}w} \otimes v , T_{{x^+}w} \otimes v \ra =\la v, v \ra$ and hence the induced form on $\Ind{\He_1}{\He_2} V = \bigoplus  V^g$ in this basis is the direct sum of $X_2 / X_1$ copies of the form on $V$.
\end{proof}

\bibliographystyle{aalpha}
\bibliography{bibfile}

@article{DelormeOpdam,
url = {https://doi.org/10.1515/crelle.2011.064},
title = {{Analytic R-groups of affine Hecke algebras}},
title = {},
author = {Patrick Delorme and Eric Opdam},
pages = {133--172},
volume = {2011},
number = {658},
journal = {Journal für die reine und angewandte Mathematik},
doi = {doi:10.1515/crelle.2011.064},
year = {2011},
lastchecked = {2026-10-13}
}

@article{lusztigUnipotent1,
    author = {Lusztig, George},
    title = { Classification of unipotent representations of simple p-adic groups },
    journal = {International Mathematics Research Notices},
    volume = {1995},
    number = {11},
    pages = {517-589},
    year = {1995},
    month = {01},
    issn = {1073-7928},
    doi = {10.1155/S1073792895000353},
    url = {https://doi.org/10.1155/S1073792895000353},
    eprint = {https://academic.oup.com/imrn/article-pdf/1995/11/517/6768457/1995-11-517.pdf},
}

@Inbook{RamRam2003,
author="Ram, Arun
and Ramagge, Jacqui",
editor="Lakshmibai, V.
and Balaji, V.
and Mehta, V. B.
and Nagarajan, K. R.
and Pranjape, K.
and Sankaran, P.
and Sridharan, R.",
title="Affine {Hecke} Algebras, Cyclotomic {Hecke} algebras and Clifford Theory",
bookTitle="A Tribute to C. S. Seshadri: Perspectives in Geometry and Representation Theory",
year="2003",
publisher="Hindustan Book Agency",
address="Gurgaon",
pages="428--466",
isbn="978-93-86279-11-8",
doi="10.1007/978-93-86279-11-8_26",
}

@misc{AntorOkada,
      title={A Microlocal Description of {Aubert-Zelevinsky} Duality on Unipotent {$L$}-Parameters}, 
      author={Jonas Antor and Emile Okada},
      year={2026},
      eprint={2605.06180},
      archivePrefix={arXiv},
      primaryClass={math.RT},
      url={https://arxiv.org/abs/2605.06180}, 
}

@article{Evens,
 ISSN = {00029939, 10886826},
 URL = {http://www.jstor.org/stable/2161778},
 author = {Sam Evens},
 journal = {Proceedings of the American Mathematical Society},
 number = {4},
 pages = {1285--1290},
 publisher = {American Mathematical Society},
 title = {The {Langlands} Classification for Graded {{Hecke}} Algebras},
 urldate = {2026-06-23},
 volume = {124},
 year = {1996}
}

@article{Reeder2002,
  title     = {Isogenies of {{Hecke}} algebras and a {Langlands} correspondence for ramified principal series representations},
  author    = {Reeder, Mark},
  journal   = {Representation Theory of the American Mathematical Society},
  volume    = {6},
  pages     = {101--126},
  year      = {2002},
  publisher = {American Mathematical Society},
  doi       = {10.1090/S1088-4165-02-00167-X},
  url       = {https://www.ams.org/ert/2002-006-04/S1088-4165-02-00167-X/}
}

@article{Solleveld_2021,
   title={Affine {Hecke} algebras and their representations},
   volume={32},
   ISSN={0019-3577},
   url={http://dx.doi.org/10.1016/j.indag.2021.01.005},
   DOI={10.1016/j.indag.2021.01.005},
   number={5},
   journal={Indagationes Mathematicae},
   publisher={Elsevier BV},
   author={Solleveld, Maarten},
   year={2021},
 pages={1005–1082} }

@article{kato1981irreducibility,
  title={Irreducibility of principal series representations for {{Hecke}} algebras of affine type},
  author={Kato, {Shin-ichi}},
  journal={J. Fac. Sci. Univ. Tokyo Sect. IA Math},
  volume={28},
  number={3},
  pages={929--43},
  year={1981}
}

@Book{BorelWallach,
  title     = {Continuous cohomology, discrete subgroups, and representations of reductive groups},
  publisher = {American Mathematical Society},
  year      = {2000},
  author    = {Borel, Armand. and Wallach, Nolan R.},
  series    = {Mathematical surveys and monographs, v. 67},
  address   = {Providence, R.I},
  edition   = {2nd ed.},
  isbn      = {0821808516},
  booktitle = {Continuous cohomology, discrete subgroups, and representations of reductive groups},
  language  = {eng},
  lccn      = {98044527},
}

@article{Solleveld2010,
  title={On the classification of irreducible representations of affine {Hecke} algebras with unequal parameters},
  author={Maarten Solleveld},
  journal={Representation Theory of The American Mathematical Society},
  year={2012},
  volume={16},
  pages={1-87},
  url={https://api.semanticscholar.org/CorpusID:119140027}
}

@article{EM,
author = {Evens, Sam and Mirkovi\'c, Ivan},
year = {1997},
month = {02},
pages = {},
title = {{Fourier transform and the Iwahori-Matsumoto involution}},
volume = {86},
journal = {Duke Mathematical Journal},
doi = {10.1215/S0012-7094-97-08613-0}
}

@misc{BC2013,
      title={{Hermitian forms for affine {Hecke} algebras}}, 
note = {v1} ,
       author={Dan Barbasch and Dan Ciubotaru},
      year={2013},
howpublished={https://arxiv.org/abs/1312.3316}
 }

@article{RamKril,
author = {Kriloff, Cathy and Ram, Arun},
year = {2002},
month = {04},
pages = {31-69},
title = {{Representations of graded {Hecke} algebras}},
volume = {6},
journal = {Representation Theory of The American Mathematical Society},
doi = {10.1090/S1088-4165-02-00160-7}
}

@book{Humphreys_1990, place= "Cambridge", series={Cambridge Studies in Advanced Mathematics}, title={Reflection Groups and Coxeter Groups}, publisher={Cambridge University Press}, author={Humphreys, James E.}, year={1990}, collection={Cambridge Studies in Advanced Mathematics}}

@article{Solleveld_Opdam_2026, title={HERMITIAN DUALS AND GENERIC REPRESENTATIONS FOR AFFINE {H}ECKE ALGEBRAS}, volume={25}, DOI={10.1017/S1474748026101790}, number={4}, journal={Journal of the Institute of Mathematics of Jussieu}, author={Solleveld, Maarten and Opdam, Eric}, year={2026}, pages={1997–2045}}

@article{Lusztig,
  title={{Affine {Hecke} algebras and their graded version}},
  author={George Lusztig},
  journal={Journal of the American Mathematical Society},
  year={1989},
  volume={2},
  pages={599-635},
  url={https://api.semanticscholar.org/CorpusID:121133856}
}

\end{document}